\documentclass[amssymb, 11pt]{amsart}
\usepackage{latexsym}
\usepackage{young}
\usepackage{hyperref}

\newlength{\standardunitlength}
\newcommand{\bea}{\begin{eqnarray}}
\newcommand{\ena}{\end{eqnarray}}
\newcommand{\beas}{\begin{eqnarray*}}
\newcommand{\enas}{\end{eqnarray*}}

\newcommand{\ignore}[1]{}

\newtheorem{prop}{Proposition}[section]

\newtheorem{lemma}[prop]{Lemma}

\newtheorem{theorem}[prop]{Theorem}

\begin{document}

\title [A generating function] {A generating function approach to counting theorems for square-free
polynomials and maximal tori}

\author{Jason Fulman}

\address{Department of Mathematics\\
        University of Southern California\\
        Los Angeles, CA, 90089}
\email{fulman@usc.edu}

\thanks{Fulman was supported by NSA grant H98230-13-1-0219. We thank Robert Guralnick for useful discussions.}

\keywords{generating function, square-free polynomial, maximal torus}

\date{October 13, 2014}

\begin{abstract} A recent paper of Church, Ellenberg, and Farb uses topology and representation theory of the symmetric group
to prove enumerative results about square-free polynomials and $F$-stable maximal tori of $GL_n(\overline{F_q})$. In this note,
we use generating functions to give elementary proofs of some of their results, and some extensions.
\end{abstract}

\maketitle

\section{Introduction} \label{into}

A recent paper by Church, Ellenberg, and Farb \cite{CEF} (building on themes of Lehrer \cite{Le}, \cite{L3}) relates topology and representation theory to enumerative problems about varieties over finite fields, with a useful flow of information in both directions. They develop counting theorems for random square-free monic degree $n$ polynomials over a finite field $F_q$, and for $F$-stable maximal tori of $GL_n(\overline{F_q})$; here $F$ denotes the Frobenius map. They prove the following results using topology and representation theory.

\begin{center}
Results for square-free polynomials
\end{center}

\begin{enumerate}

\item The number of monic degree $n$ square-free polynomials is $q^n-q^{n-1}$.

\item The expected number of linear factors of a random monic degree $n$ square-free polynomial is
\[ 1 - \frac{1}{q} + \frac{1}{q^2} - \frac{1}{q^3} + \cdots \pm \frac{1}{q^{n-2}} .\]

\item The expected excess of irreducible vs. reducible quadratic factors tends to
\[ \frac{1}{q}-\frac{3}{q^2}+\frac{4}{q^3}-\frac{4}{q^4}+\frac{5}{q^5}-\frac{7}{q^6}+\frac{8}{q^7}-\frac{8}{q^8} + \cdots   \] as $n \rightarrow \infty$.

\item The discriminant of a random monic degree $n$ square-free polynomial is equidistributed in $F_q^*$ between residues and nonresidues

\end{enumerate}

Of these results, the first is well known and elementary; see for example \cite{FGP} and the references therein.

\begin{center}
Results for $F$-stable maximal tori of $GL_n(\overline{F_q})$.
\end{center}

\begin{itemize}

\item The number of $F$-stable maximal tori of $GL_n(\overline{F_q})$ is $q^{n^2-n}$.

\item The expected number of eigenvectors of a random $F$-stable maximal torus of $GL_n(\overline{F_q})$ is
\[ 1 + \frac{1}{q} + \frac{1}{q^2} + \cdots + \frac{1}{q^{n-1}} \]

\item The expected excess of reducible vs. irreducible dimension 2 subtori tends to
\[ \frac{1}{q} + \frac{1}{q^2} + \frac{2}{q^3} + \frac{2}{q^4} + \frac{3}{q^5} + \frac{3}{q^6} + \frac{4}{q^7} + \frac{4}{q^8} + \cdots \]
as $n \rightarrow \infty$.

\item The number of irreducible factors is more likely to be $\equiv$ n mod 2 than not, with bias equal to the square root of the
number of tori

\end{itemize}

Of these results, the first is a theorem of Steinberg \cite{St}, proved using the Grothendieck-Lefschetz formula by
Srinivasan \cite{S} and Lehrer \cite{Le}. The other results may be new, and according to \cite{CEF}, ``not so easy to prove''.

The main goal of the current paper is to prove all of the above results using generating functions. The use of generating functions
to study square-free polynomials is not new (see \cite{FGP}), though some of our arguments are. We do obtain one new result; a ``finite n'' formula
for the expected excess of irreducible vs. reducible quadratic factors. The use of generating functions to study $F$-stable maximal tori of
$GL_n(\overline{F_q})$ does appear to be new, and gives elementary proofs for all of the above results from \cite{CEF}. We do obtain one new result:
a ``finite n'' formula for the expected excess of reducible vs. irreducible dimension 2 subtori. Our applications of generating function methods
to random maximal tori only scratch the surface, and we expect that there will be many more applications, such as finding analogs for maximal tori
of results in Section 4 of \cite{CEF}. We also note that our methods for studying maximal tori will generalize to the other classical groups
(unitary, symplectic, orthogonal), since for all such groups there are nice product formulas for the order of tori and for centralizer sizes in the corresponding Weyl group.

To close the introduction, we note that generating function techniques for enumerative problems over finite fields can be very powerful. For example, it
is proved (independently in \cite{F} and \cite{W}) using generating functions that:
\begin{enumerate}
\item The $n \rightarrow \infty$ limiting proportion of elements of $GL(n,q)$ with square-free characteristic polynomial is $1-1/q$.
\item The $n \rightarrow \infty$ limiting proportion of cyclic elements of $GL(n,q)$ (that is elements whose characteristic polynomial is equal to its minimal polynomial) is $(1-1/q^5)/(1+1/q^3)$.
\end{enumerate} At the current time no other proofs of these two results are known.

\section{Square free polynomials} \label{squarefree}

The purpose of this section is to use generating functions to study random square-free polynomials over a finite field, proving the results of \cite{CEF} stated in the introduction. It is known that generating functions can be used to study random square-free polynomials (see for instance \cite{FGP}). However some of the methods of calculations in the proofs of Theorems \ref{lin} and \ref{exc} may be new, and we do obtain a ``finite n'' version of a result from \cite{CEF}. This section serves as a useful ``warm-up'' for the more subtle results in Section \ref{maximal}.

In what follows we let $N(d,q)$ denote the number of monic irreducible degree $d$ polynomials over the field $F_q$. The following lemma will be useful throughout this section.

\begin{lemma} \label{prod} \[ \prod_{d \geq 1} (1-u^d)^{-N(d,q)} = \frac{1}{1-qu} .\]
\end{lemma}

\begin{proof} By uniqueness of factorization of polynomials over the field $F_q$, it follows that the coefficient of $u^n$ on the left hand side of the lemma is equal to the total number of monic degree $n$ polynomials over $F_q$, which is $q^n$. As this is equal to the coefficient of $u^n$ on the right hand side of the lemma, the result follows.
\end{proof}

As a first example, we use generating functions to enumerate the square-free monic degree $n$ polynomials with coefficients in $F_q$. This application is known (see for instance \cite{FGP}), but we include it for expository purposes.

\begin{theorem} \label{total} For $n \geq 2$, the number of square-free monic degree $n$ polynomials with coefficients in $F_q$ is equal to $q^n-q^{n-1}$.
\end{theorem}

\begin{proof} The number of square-free monic degree $n$ polynomials is clearly the coefficient of $u^n$ in
\[ \prod_{d \geq 1} (1+u^d)^{N(d,q)} .\] Note by Lemma \ref{prod} that
\[ \prod_{d \geq 1} (1+u^d)^{N(d,q)} = \frac{ \prod_{d \geq 1} (1-u^{2d})^{N(d,q)}}{ \prod_{d \geq 1} (1-u^{d})^{N(d,q)}} = \frac{1-u^2q}{1-uq} .\] It is easily seen that for $n \geq 2$ the coefficient of $u^n$ in $(1-u^2q)/(1-uq)$ is equal to $q^n-q^{n-1}$.
\end{proof}

As mentioned in the introduction, the paper \cite{CEF} computed the expected number of linear factors of a random monic degree $n$ square-free polynomial.
Theorem \ref{lin} derives this using generating functions.

\begin{theorem} \label{lin} For $n \geq 2$, the expected number of linear factors of a random monic degree $n$ square-free polynomial is
\[ 1 - \frac{1}{q} + \frac{1}{q^2} - \frac{1}{q^3} + \cdots \pm \frac{1}{q^{n-2}} .\]
\end{theorem}

\begin{proof} Let $P(n,q)$ denote the set of monic degree $n$ square-free polynomials over $F_q$, and for $P \in P(n,q)$, let $n_1(P)$ denote the number of
linear factors of $P$. Then
\[ \sum_{P \in P(n,q)} x^{n_1(P)} \] is equal to the coefficient of $u^n$ in
\begin{eqnarray*}
(1+xu)^{N(1,q)} \prod_{d \neq 1} (1+u^d)^{N(d,q)} & = & \frac{(1+xu)^q}{(1+u)^q} \prod_{d \geq 1} (1+u^d)^{N(d,q)} \\
& = & \frac{(1+xu)^q}{(1+u)^q} \frac{1-u^2q}{1-uq},
\end{eqnarray*} where the second equality followed by arguing as in Theorem \ref{total}. To compute the expected value of $n_1$, we differentiate
with respect to $x$, set $x=1$, then take the coefficient of $u^n$, and divide by $q^n-q^{n-1}$.

Differentiating with respect to $x$ and setting $x=1$ yields \[ \frac{qu}{(1+u)} \frac{(1-u^2q)}{1-uq} .\] Taking the coefficient of $u^n$ and dividing by $q^n(1-1/q)$ gives the coefficient of $u^n$ in \[ \frac{1}{1-1/q} \frac{u}{1+u/q} \frac{1-u^2/q}{1-u} .\] This is easily seen to be equal to
\[ 1 - \frac{1}{q} + \frac{1}{q^2} - \frac{1}{q^3} + \cdots \pm \frac{1}{q^{n-2}}, \] proving the theorem.
\end{proof}

As stated in the introduction, the paper \cite{CEF} computed the expected excess of irreducible versus reducible quadratic factors of a random monic degree $n$ square-free polynomial, in the limit that $n \rightarrow \infty$. We recover this result using generating functions, and also determine an exact formula for finite $n$.

\begin{theorem} \label{exc}
\begin{enumerate}

\item The expected excess of irreducible versus reducible quadratic factors of a random monic degree $n$ square-free polynomial
tends to \[ \frac{1}{q} \frac{1-1/q}{(1+1/q)^2(1+1/q^2)} \] as $n \rightarrow \infty$.

\item The expected excess of irreducible versus reducible quadratic factors of a random monic degree $n$ square-free polynomial is equal to $0$ for $n=2$, $1/q$ for $n=3$, and for $n \geq 4$ is \[ \left[ \sum_{i=1}^{n-3} \frac{(-1)^{i+1} a_i}{q^i} \right] - \frac{(-1)^n b_{n-3}}{q^{n-2}} .\] Here $a_1,a_2,a_3,\cdots $ is the sequence
\[ 1,3,4,4,5,7,8,8,9,11,12,12,13,15,16,16, \cdots \] of odd numbers with two times the positive even numbers
repeated in order between them, and $b_1,b_2,b_3,\cdots $ is the sequence
\[ 2,2,2,3,4,4,4,5,6,6,6,7,8,8,8,9,10,10,10 \cdots \] of three even numbers followed by one odd number.

\end{enumerate}

\end{theorem}

\begin{proof} Let $P(n,q)$ denote the set of monic degree $n$ square-free polynomials over $F_q$, and for $P \in P(n,q)$, let $n_i(P)$ denote the number of
degree $i$ factors of $P$. We need to computed the expected value of $n_2(P) - {n_1(P) \choose 2}$.

Arguing as in the proof of Theorem \ref{lin} shows that \[ \sum_{P \in P(n,q)} x^{n_2(P)} \] is equal to the coefficient of $u^n$ in
\[ \frac{(1+xu^2)^{(q^2-q)/2}}{(1+u^2)^{(q^2-q)/2}} \frac{1-u^2q}{1-uq} .\] To compute the expected value of $n_2$, one must differentiate with respect to $x$, set $x=1$, take the coefficient of $u^n$ and then divide by $q^n (1-1/q)$. One concludes that $E[n_2]$ is equal to the coefficient of $u^n$ in
\[ \frac{1}{1+u^2} \frac{q^2 u^2}{2q^n} \frac{(1-u^2q)}{(1-uq)} .\]

Next, recall from the proof of Theorem \ref{lin} that \[ \sum_{P \in P(n,q)} x^{n_1(P)} \] is equal to the coefficient of
$u^n$ in \[ \frac{(1+xu)^q}{(1+u)^q} \frac{1-u^2q}{1-uq} .\] To compute $E[n_1(n_1-1)]$, one must differentiate twice with respect to $x$, set $x=1$, take the coefficient of $u^n$ and then divide by $q^n (1-1/q)$. One concludes that $E[n_1(n_1-1)]$, is equal to the coefficient of $u^n$ in
\[ \frac{1}{(1+u)^2} \frac{q^2 u^2}{q^n} \frac{(1-u^2q)}{(1-uq)} .\]

Combining the previous two paragraphs, it follows that the expected value of $n_2(P) - {n_1(P) \choose 2}$ is equal to the coefficient of $u^n$
in \[ \left[ \frac{1}{1+u^2} - \frac{1}{(1+u)^2} \right] \frac{q^2 u^2}{2q^n} \frac{(1-u^2q)}{(1-uq)}.\] This is equal to the coefficient of $u^n$
in \begin{equation} \label{genf} \left[ \frac{1}{1+(u/q)^2} - \frac{1}{(1+u/q)^2} \right] \frac{u^2}{2} \frac{(1-u^2/q)}{1-u} .\end{equation} If a Taylor series of
a function $f$ around $0$ converges at $u=1$, then the $n \rightarrow \infty$ limit of the coefficient of $u^n$ in $f(u)/(1-u)$ is equal to $f(1)$. Thus the $n \rightarrow \infty$ limit of the expected value of $n_2(P) - {n_1(P) \choose 2}$ is equal to
\[ \left[ \frac{1}{1+(1/q)^2} - \frac{1}{(1+1/q)^2} \right] \frac{1}{2} (1-1/q) =  \frac{1}{q} \frac{1-1/q}{(1+1/q)^2(1+1/q^2)}, \]
and the first part of the theorem is proved.

To prove the second assertion, it is necessary to show that
\[ \frac{u^3}{q} + \sum_{n \geq 4} \left( \left[ \sum_{i=1}^{n-3} \frac{(-1)^{i+1} a_i}{q^i} \right] - \frac{(-1)^n b_{n-3}}{q^{n-2}} \right) u^n \] is equal to \eqref{genf}.

Now one computes that
\begin{eqnarray*}
\sum_{n \geq 4} u^n \sum_{i=1}^{n-3} \frac{(-1)^{i+1} a_i}{q^i} & = & \sum_{i \geq 1} \frac{(-1)^{i+1} a_i}{q^i} \sum_{n \geq i+3} u^n \\
& = & \sum_{i \geq 1} \frac{(-1)^{i+1} a_i u^{i+3}}{q^i (1-u)} \\
& = & \frac{-u^3}{1-u} \sum_{i \geq 1} (-u/q)^i a_i.
\end{eqnarray*} One easily checks that
\[ \sum_{i \geq 1} u^i a_i = \frac{u(1+u)}{(1-u)^2 (1+u^2)}, \] which implies that
\begin{equation} \label{fi}
\sum_{n \geq 4} u^n \sum_{i=1}^{n-3} \frac{(-1)^{i+1} a_i}{q^i} = \frac{u^4 (1-u/q)}{q (1-u)(1+u/q)^2 (1+u^2/q^2)}
\end{equation}

Next one computes that
\[ - \sum_{n \geq 4} \frac{(-1)^n b_{n-3} u^n}{q^{n-2}} = -u^3 \sum_{i \geq 1} \frac{(-1)^{i+3} b_i u^i}{q^{i+1}} = \frac{u^3}{q} \sum_{i \geq 1} (-u/q)^i b_i.\] One easily checks that \[ \sum_{i \geq 1} u^i b_i = \frac{1}{(1-u)^2 (1+u^2)} - 1, \] which implies that
\begin{equation} \label{se}
- \sum_{n \geq 4} \frac{(-1)^n b_{n-3} u^n}{q^{n-2}} = \frac{u^3}{q} \left[ \frac{1}{(1+u/q)^2 (1+u^2/q^2)} - 1 \right].
\end{equation}

Combining \eqref{fi} and \eqref{se} completes the proof.
\end{proof}

To close this section, we give a generating function proof of the fact from \cite{CEF} that the discriminant of a random square-free polynomial is
equidistributed in $F_q^*$ between residues and nonresidues.

\begin{prop} For $n \geq 2$, the discriminant of a random monic degree $n$ square-free polynomial is
equidistributed in $F_q^*$ between residues and nonresidues.
\end{prop}

\begin{proof} We can assume that $q$ is odd. As noted in \cite{CEF}, it suffices to show that if $\mu(P)=(-1)^d$, where $d$ is the number of irreducible factors in a square-free
polynomial $P$, then $\mu(P)=1$ exactly half the time. Observe that
\[ \sum_{P \in P(n,q)} \mu(P) \] is equal to the coefficient of $u^n$ in
\[ \prod_{d \geq 1} (1-u^d)^{N(d,q)}.\] By Lemma \ref{prod} this is the coefficient of $u^n$ in $1-uq$, which is $0$ for $n \geq 2$.
\end{proof}

\section{Maximal tori} \label{maximal}

For background on maximal tori, we recommend Chapter 3 of \cite{C}. We do point out that $F$-stable maximal tori of $GL_n(\overline{F_q})$ are not the same as maximal tori of $GL_n(q)$; indeed there are four $F$-stable maximal tori of $GL_2(\overline{F_2})$, but only two maximal tori of $GL_2(2)$. In this section we using generating functions to study random $F$-stable maximal tori of $GL_n(\overline{F_q})$. There is a map from the set of such tori to the conjugacy classes of the symmetric group $S_n$. The conjugacy classes of $S_n$ are parameterized by partitions $\lambda$ of $n$, and we say that a torus mapping to the partition $\lambda$ has type $\lambda$. Lemma \ref{form} enumerates the number of tori of type $\lambda$, which will be crucial to our approach.

\begin{lemma} \label{form} Let $\lambda$ be a partition of $n$, and let $n_i$ denote the number of parts of $\lambda$ of size $i$. Then the number of $F$-stable maximal tori of $GL_n(\overline{F_q})$ of type $\lambda$ is equal to
\[ \frac{|GL(n,q)|}{\prod_i i^{n_i} n_i! \prod_i (q^i-1)^{n_i}}.\]
\end{lemma}

\begin{proof} This is immediate from Section 2.7 of \cite{SS}, together with the fact that the centralizer size of an element of $S_n$ of type $\lambda$ is equal to $\prod_i i^{n_i} n_i!$. \end{proof}

{\it Remark:} A special case of Lemma \ref{form} is that the number of irreducible tori is equal to \[ \frac{q^{{n \choose 2}}}{n} (q-1) (q^2-1) \cdots (q^{n-1}-1). \] For an interesting proof of this using topology and representation theory, see Proposition 5.12 of \cite{CEF}.

As a consequence of Lemma \ref{form}, one has the following generating function, which is similar to the ``cycle index'' of the symmetric groups. We let $T(n,q)$ denote the set of $F$-stable maximal tori $T$ of $GL_n(\overline{F_q})$, and let $n_i(T)$ denote the number of parts of size $i$ of the partition corresponding to $T$.

\begin{theorem} \label{cyc}
\[ 1 + \sum_{n \geq 1} \frac{u^n}{|GL(n,q)|} \sum_{T \in T(n,q)} \prod_i x_i^{n_i(T)} =
\prod_{i \geq 1} \exp \left[ \frac{x_i u^i}{(q^i-1) i} \right]. \]
\end{theorem}

\begin{proof} Lemma \ref{form} implies that the left-hand side of the theorem is equal to
\[ 1 + \sum_{n \geq 1} u^n \sum_{|\lambda|=n} \frac{\prod_i x_i^{n_i}}{\prod_i i^{n_i} n_i! (q^i-1)^{n_i}} .\]
By the Taylor expansion of the exponential function, this is equal to \[ \prod_{i \geq 1} \exp \left[ \frac{x_i u^i}{(q^i-1) i} \right] .\]
\end{proof}

The following result of Euler, which is a special case of Corollary 2.2 of \cite{A}, will be helpful.

\begin{lemma} \label{euler}
\begin{enumerate}
\item \[ \prod_{r \geq 1} \frac{1}{1-u/q^r} = 1 + \sum_{n \geq 1} \frac{u^n}{q^n (1-1/q) \cdots (1-1/q^n)} \]

\item \[ \prod_{r \geq 1} (1+u/q^r) = 1 + \sum_{n \geq 1} \frac{u^n}{q^{{n+1 \choose 2}} (1-1/q) \cdots (1-1/q^n)} \]
\end{enumerate}
\end{lemma}

Theorem \ref{count} proves that the total number of $F$-stable maximal tori of $GL_n(\overline{F_q})$ is equal to $q^{n^2-n}$. This formula is due to Steinberg \cite{St}; proofs using the Grothendieck-Lefschetz formula have been given by Srinivasan \cite{S} and Lehrer \cite{Le}. We give a proof using the generating function of Theorem \ref{cyc}.

\begin{theorem} \label{count} The number of $F$-stable maximal tori of $GL_n(\overline{F_q})$ is equal to $q^{n^2-n}$.
\end{theorem}

\begin{proof} It follows from setting all the $x_i=1$ in Theorem \ref{cyc} that the number of $F$-stable maximal tori of $GL_n(\overline{F_q})$ is equal to
$|GL(n,q)|$ multiplied by the coefficient of $u^n$ in
\begin{eqnarray*}
\prod_{i \geq 1} \exp \left( \frac{1}{(q^i-1)} \frac{u^i}{i} \right) & = & \prod_{i \geq 1} \exp \left( \frac{1}{q^i (1-1/q^i)} \frac{u^i}{i} \right) \\
& = & \prod_{i \geq 1} \exp \left( \frac{u^i}{i} \sum_{r \geq 1} \frac{1}{q^{ir}} \right) \\
& = & \prod_{r \geq 1} \prod_{i \geq 1} \exp \left( (u/q^r)^i / i \right) \\
& = & \prod_{r \geq 1} \exp \left( - \log(1-u/q^r) \right) \\
& = & \prod_{r \geq 1} \frac{1}{1-u/q^r}.
\end{eqnarray*} Using part 1 of Lemma \ref{euler}, it follows that the total number of maximal tori of $GL(n,q)$ is equal to
\[ \frac{|GL(n,q)|}{q^n (1-1/q) \cdots (1-1/q^n)} = q^{n^2-n} .\]
\end{proof}

{\it Remark:} One can pick an $F$-stable maximal torus of $GL_n(\overline{F_q})$ uniformly at random, and let $\lambda$ be the partition of $n$ corresponding to its
type. From Lemma \ref{form} and Theorem \ref{count}, it follows that the resulting random partition is equal to a partition $\lambda$ of $n$ with probability \[ \frac{(1-1/q) \cdots (1-1/q^n)}{\prod_i i^{n_i} n_i! \prod_i (1-1/q^i)^{n_i}} .\] This measure is a special case of the random partitions studied in \cite{DR}. The fact that
\[ \sum_{|\lambda|=n} \frac{(1-1/q) \cdots (1-1/q^n)}{\prod_i i^{n_i} n_i! \prod_i (1-1/q^i)^{n_i}} = 1 \] is an identity which Andrews (page 81 of \cite{A}) attributes to Cayley.

As noted in the introduction, Church, Ellenberg, and Farb \cite{CEF} calculate the expected number of eigenvectors of a random $F$-stable maximal torus of $GL_n(\overline{F_q})$. Here by ``number of eigenvectors'' of $T$ is meant the number of lines in the projective space $P^{n-1}(F_q)$ fixed by $T$. Theorem \ref{eigen} uses generating functions to derive their formula.

\begin{theorem} \label{eigen} The expected number of eigenvectors of a random $F$-stable maximal torus of $GL_n(\overline{F_q})$ is equal to
 \[ 1 + \frac{1}{q} + \frac{1}{q^2} + \cdots + \frac{1}{q^{n-1}} .\]
\end{theorem}

\begin{proof} As noted in \cite{CEF}, the number of eigenvectors of an $F$-stable maximal torus $T$ of is equal to $n_1(T)$.
Set $x_1=x$ and $x_i=1$ for $i \geq 2$ in Theorem \ref{cyc}. It follows that
\begin{eqnarray*}
1 + \sum_{n \geq 1} \frac{u^n}{|GL(n,q)|} \sum_{T \in T(n,q)} x^{n_1(T)} & = & \exp \left[ \frac{xu}{q-1} \right] \prod_{i \geq 2} \exp \left[ \frac{u^i}{(q^i-1)i} \right] \\
& = & \exp \left[ \frac{(x-1)u}{q-1} \right] \prod_{i \geq 1} \exp \left[ \frac{u^i}{(q^i-1)i} \right] \\
& = & \exp \left[ \frac{(x-1)u}{q-1} \right] \prod_{r \geq 1} \frac{1}{1-u/q^r},
\end{eqnarray*} where the last step followed by arguing as in Theorem \ref{count}.

To compute the expected value of $n_1(T)$, one must differentiate with respect to x, set $x=1$, take the coefficient of $u^n$, and then multiply by $|GL(n,q)|/q^{n^2-n}$. Differentiating with respect to $x$ and setting $x=1$ yields
\[ \frac{u}{q-1} \prod_{r \geq 1} \frac{1}{1-u/q^r} .\] By part 1 of Lemma \ref{euler}, taking the coefficient of $u^n$ yields
\[ \frac{1}{q-1} \frac{1}{q^{n-1} (1-1/q) \cdots (1-1/q^{n-1})}.\] Multiplying this by $|GL(n,q)|/q^{n^2-n}$ yields
\[ \frac{1-1/q^n}{1-1/q} = 1 + \frac{1}{q} + \frac{1}{q^2} + \cdots + \frac{1}{q^{n-1}}, \] as needed.
\end{proof}

As stated in the introduction, the paper \cite{CEF} computes the expected excess of reducible vs. irreducible dimension two subtori of a random maximal torus. They show that as $n \rightarrow \infty$, this quantity tends to \[ \frac{1}{q} \frac{1}{(1-1/q)(1-1/q^2)} .\] Theorem \ref{excess}
extends this to finite $n$.

\begin{theorem} \label{excess} Let $T$ be a random $F$-stable maximal torus of $GL_n(\overline{F_q})$. Then for $n \geq 2$, the expected value of the number of reducible dimension two subtori minus the number of irreducible dimension two subtori is equal to
\[ \frac{1}{q} \frac{(1-1/q^{n-1}) (1-1/q^n)}{(1-1/q)(1-1/q^2)} .\]
\end{theorem}

\begin{proof} The number of reducible dimension two subtori is equal to ${n_1(T) \choose 2}$, and the number of irreducible dimension two subtori is equal to $n_2(T)$, so it is necessary to compute the expected value of \[ {n_1(T) \choose 2} - n_2(T).\]

First we compute the expected value of $n_1(n_1-1)$. Recall from the proof of Theorem \ref{eigen} that
\[ 1 + \sum_{n \geq 1} \frac{u^n}{|GL(n,q)|} \sum_{T \in T(n,q)} x^{n_1(T)} = \exp \left[ \frac{(x-1)u}{q-1} \right] \prod_{r \geq 1} \frac{1}{1-u/q^r}.\] To compute the expected value of $n_1(n_1-1)$, one must differentiate twice with respect to x, set $x=1$, take the coefficient of $u^n$, and then multiply by $|GL(n,q)|/q^{n^2-n}$. Differentiating twice with respect to $x$ and setting $x=1$ yields \[ \frac{u^2}{(q-1)^2} \prod_{r \geq 1} \frac{1}{1-u/q^r}.\] By part 1 of Lemma \ref{euler}, taking the coefficient of $u^n$ yields
\[ \frac{1}{(q-1)^2} \frac{1}{q^{n-2} (1-1/q) \cdots (1-1/q^{n-2})}.\] Multiplying this by $|GL(n,q)|/q^{n^2-n}$ yields
\[ \frac{q^2 (1-1/q^{n-1}) (1-1/q^n)}{(q-1)^2}. \] Thus
\begin{equation} \label{fir} E \left[ {n_1(T) \choose 2} \right] = \frac{q^2 (1-1/q^{n-1}) (1-1/q^n)}{2 (q-1)^2}. \end{equation}

Next we compute the expected value of $n_2$. In Theorem \ref{cyc}, set $x_2=x$ and all other $x_i=1$. It follows that
\begin{eqnarray*}
1 + \sum_{n \geq 1} \frac{u^n}{|GL(n,q)|} \sum_{T \in T(n,q)} x^{n_2(T)} & = & \exp \left[ \frac{xu^2}{2(q^2-1)} \right] \prod_{i \neq 2}
\exp \left[ \frac{u^i}{(q^i-1)i} \right] \\
& = & \exp \left[ \frac{(x-1) u^2}{2(q^2-1)} \right] \prod_{i \geq 1}
\exp \left[ \frac{u^i}{(q^i-1)i} \right] \\
& = & \exp \left[ \frac{(x-1) u^2}{2(q^2-1)} \right] \prod_{r \geq 1} \frac{1}{1-u/q^r},
\end{eqnarray*} where the last step is from the proof of Theorem \ref{count}. To compute the expected value of $n_2$, one must differentiate with respect to x, set $x=1$, take the coefficient of $u^n$, and then multiply by $|GL(n,q)|/q^{n^2-n}$. Differentiating with respect
to $x$ and setting $x=1$ yields \[ \frac{u^2}{2(q^2-1)} \prod_{r \geq 1} \frac{1}{1-u/q^r} .\] By part 1 of Lemma \ref{euler},
taking the coefficient of $u^n$ yields \[ \frac{1}{2(q^2-1)} \frac{1}{q^{n-2} (1-1/q) \cdots (1-1/q^{n-2})}.\] Multiplying this by $|GL(n,q)|/q^{n^2-n}$, it follows that
\begin{equation} \label{sec} E[n_2(T)] = \frac{q^2}{2(q^2-1)} (1-1/q^{n-1})(1-1/q^n). \end{equation}

The result now follows by combining equations \eqref{fir} and \eqref{sec}.
\end{proof}

As stated in the introduction, the paper \cite{CEF} shows that the number of irreducible factors of an $F$-stable maximal torus of $GL_n(\overline{F_q})$ is more likely to be $\equiv$ n mod 2 than not, with bias equal to $q^{(n^2-n)/2}$ (which is the square root of the number of tori). Theorem \ref{mod} proves this using generating functions.

\begin{theorem} \label{mod} The number of $F$-stable maximal tori of $GL_n(\overline{F_q})$ with number of irreducible factors $\equiv$ n mod 2, minus the number of $F$-stable maximal tori of $GL_n(\overline{F_q})$ with number of irreducible factors $\not \equiv$ n mod 2, is equal to $q^{(n^2-n)/2}$.
\end{theorem}

\begin{proof} Setting all $x_i=-1$ and $u=-u$ in Theorem \ref{cyc}, it follows that the sought quantity is equal to $|GL(n,q)|$ multiplied by the coefficient of $u^n$ in \[ \prod_{i \geq 1} \exp \left[ - \frac{(-u)^i}{(q^i-1)i} \right] = \prod_{r \geq 1} (1+u/q^r),\] where the equality is proved by arguing as in the proof of Theorem \ref{count}. The result now follows from part 2 of Lemma \ref{euler}.
\end{proof}

\end{document}